\documentclass[11pt]{amsart}

\usepackage{amssymb}
\usepackage{amsmath}
\usepackage{enumerate}
\usepackage[latin1]{inputenc}

\usepackage[pdftex]{graphicx}

\def\s{\mathbb{S}}
\def\h{\mathbb{H}}
\def\r{\mathbb{R}}
\def\z{\mathbb{Z}}

\def\c{\mathbb{C}}
\def\p{\mathbb{P}}
\def\m{\mathbb{M}}

\newtheorem{theorem}{Theorem}
\newtheorem{proposition}{Proposition}

\theoremstyle{definition}

\theoremstyle{remark}
  \newtheorem{remark}{Remark}

\numberwithin{equation}{section}

\DeclareMathOperator{\arctanh}{arctanh}

\DeclareMathOperator{\arcsinh}{arcsinh}

\begin{document}

\title[Compact stable CMC surfaces in the Berger spheres]{Compact stable constant mean curvature surfaces in the Berger spheres}

\author{Francisco Torralbo}
\address{Departamento de Geometr\'{\i}a  y Topolog\'{\i}a \\
Universidad de Granada \\
18071 Granada, SPAIN} \email{ftorralbo@ugr.es}

\author{Francisco Urbano}
\address{Departamento de Geometr\'{\i}a  y Topolog\'{\i}a \\
Universidad de Granada \\
18071 Granada, SPAIN} \email{furbano@ugr.es}

\thanks{Research partially supported by a MCyT-Feder research project MTM2007-61775 and a Junta Andalucía  Grant P06-FQM-01642.}



\date{}
\begin{abstract}
In the $1$-parameter family of Berger spheres $\{\s^3_{\alpha},\,\alpha>0\}$ ($\s^3_1$ is the round $3$-sphere of radius $1$) we classify the stable constant mean curvature spheres, showing that in some Berger spheres ($\alpha$ close to $0$) there are unstable constant mean curvature spheres. Also, we classify the orientable compact stable constant mean curvature surfaces in $\s^3_{\alpha},\,1/3\leq\alpha<1$ proving that they are spheres or the minimal Clifford torus in $\s^3_{1/3}$. This allows to solve the isoperimetric problem in these Berger spheres.
\end{abstract}

\maketitle
\section{Introduction}
Let $\Phi:\Sigma\rightarrow M$ be an immersion of a compact orientable surface in a three dimensional oriented Riemannian manifold $M$. The immersion $\Phi$ has constant mean curvature if and only if it is a critical point of the area functional for variations that leave constant the volume {\it enclosed} by the surface.  So, it is natural to consider the second variation operator for such an immersion $\Phi$. This operator, known as {\it the Jacobi operator } of $\Phi$ is the Schrödinger operator $L:C^{\infty}(\Sigma)\rightarrow C^{\infty}(\Sigma)$ defined by:
\[
L=\Delta+|\sigma|^2+\text{Ric}(N),
\]
where $\Delta$ is the Laplacian operator of the induced metric on $\Sigma$, $\sigma$ is the second fundamental form of $\Phi$, $N$ is a unit normal field to the immersion and $\text{Ric}$ is the Ricci curvature of $M$. Let $Q:C^{\infty}(\Sigma)\rightarrow\r$ be the quadratic form associated to $L$, defined by
\[
Q(f)=-\int_{\Sigma}fLf\,dA=\int_{\Sigma}\{|\nabla f|^2-(|\sigma|^2+\text{Ric}(N))f^2\}\,dA.
\]
In this context we say that $\Phi$ is {\it stable} if the second variation of the area is non-negative for all volume preserving variations, i.e. variations whose variational vector field $fN$ satisfies $\int_{\Sigma}f\,dA=0$. So the stability of $\Phi$ means that
\begin{equation}\label{eq:def-stability}
Q(f)\geq 0,\quad \forall\, f\in C^{\infty}(\Sigma)\quad\text{with}\quad\int_{\Sigma}f\,dA=0.
\end{equation}
The first stability results of constant mean curvature surfaces are due to Barbosa, DoCarmo and Eschenburg (\cite{BC, BCE}) where they proved that the only stable compact orientable constant mean curvature surfaces of the Euclidean $3$-space, the $3$-dimensional sphere and the $3$-dimensional hyperbolic space are the geodesic spheres (in fact they proved the result not only for surfaces but for hypersurfaces in any dimension). Later  Ritoré and Ros \cite{RR} studied stable constant mean curvature surfaces of $3$-dimensional real space forms, proving that stable constant mean curvature tori in three space forms are flat and classifying completely the orientable compact stable constant mean curvature surfaces of the three dimensional real projective space. Also  Ros \cite{R3} studied stability of constant mean curvature surfaces in $\r^3/\Gamma$, $\Gamma\subset\r^3$ being a discrete subgroup of translation with rank $k$, proving that the genus of a compact stable constant mean curvature surface in $\r^3/\Gamma$ is less than or equal to $k$.  Finally  Ros \cite{R2} improved many previous results in which the genus of an orientable compact stable constant mean curvature surface was bounded below. He proved that the genus $g$ of an orientable compact stable constant mean curvature surface of an orientable three dimensional manifold with nonnegative Ricci curvature is $g\leq 3$.

In the last years, the constant mean curvature surfaces of the homogeneous Riemannian $3$-manifolds are been deeply studied. The starting point was the work of Abresch and Rosenberg \cite{AR}, where they found out a holomorphic quadratic differential in any constant mean curvature surface of a homogeneous Riemannian $3$-manifold with isometry group of dimension $4$. The Berger spheres, the Heisenberg group, the Lie group $SL(2,\r)$ and the Riemannian products $\s^2\times\r$ and $\h^2\times\r$, where $\s^2$ and $\h^2$ are the $2$-dimensional sphere and hyperbolic plane with their standard metrics, are the most relevant examples of such homogeneous $3$-manifolds. Souam \cite{S} studied stability of compact constant mean curvature surfaces of $\s^2\times\r$ and $\h^2\times\r$. In the case of $\s^2\times\r$ he classified the compact stable constant mean curvature surfaces, proving that they are certain constant mean curvature spheres.

In this paper we study stability of compact constant mean curvature surfaces in the Berger spheres. The Berger spheres are a $1$-parameter family of metrics $\{g_\alpha\,,\,\alpha>0\}$ on the $3$-dimensional sphere which deform the standard metric $g_1$. Section $2$ describes and shows the most important properties of these $3$-manifolds. Our first result is the classification of the stable constant mean curvature spheres (Theorem \ref{thm:stability-spheres}), showing that in almost all the Berger spheres they are stable, but there are some Berger spheres (with $\alpha$ close to $0$) in which some constant mean curvature spheres are unstable. The proof of Theorem \ref{thm:stability-spheres} can be adapted to other homogeneous Riemannian $3$-manifolds. Theorem \ref{thm:stability-spheres-Nil3-Sl} proves that all the constant mean curvature spheres in the Heisenberg group and in the Lie group $\mathrm{Sl}(2,\r)$ are stable. The main idea in the proof is to realize that the quadratic forms associated to the Jacobi operators of all the constant mean curvature spheres of the Berger spheres, the Heisenberg group and $\mathrm{Sl}(2,\r)$ are the same.

Perhaps one of the most surprising facts in the paper appears in Proposition \ref{prop:stability-flat-tori}, where we construct examples of stable constant mean curvature tori in certain Berger spheres, those with $\alpha\leq 1/3$ (when $\alpha=1$, i.e. in the round sphere, only spheres can be stable \cite{BCE}). Section 6 establishes  two stability results (Theorems 6 and 7). The first one proves that for the Berger spheres with $\alpha_1\leq\alpha\leq 4/3,\, \alpha_1\approx 0.217$, the compact stable constant mean curvature surfaces must be spheres (in this case all of them are stable) or embedded tori. The second one, for the Berger spheres with $1/3\leq\alpha<1$ ($\alpha_1<1/3$), classifies completely the compact stable constant mean curvature surfaces, proving that they are spheres or $\alpha=1/3$ and the surface is the Clifford torus (which is also minimal in this Berger sphere).

One of the most important ingredients in the proof of Theorems 6 is to consider the Berger spheres as hypersurfaces of the complex projective plane or the complex hyperbolic plane (see section 2) and to use a lower bound for the Willmore functional of surfaces in these four manifolds obtained by Montiel and Urbano in \cite{MU}.  In the proof of Theorem \ref{thm:compact-CMC-stable-1/3<alpha<1}, and because the Berger spheres appearing in it ($1/3\leq\alpha<1$) are hypersurfaces of the complex projective plane, we can consider, via the first standard embedding of the complex projective plane in $\r^8$, the surfaces of these Berger spheres as surfaces of $\r^8$. So we can use the functions $X:\Sigma\rightarrow\r^8$, $X$ being a harmonic vector field in the surface, as test functions to study stability. The idea to consider harmonic vector fields to study stability had been used by Palmer \cite{P} and Ros \cite{R2,R3}.

Finally, in the last section and as an application of our results, we solve the isoperimetric problem for the Berger spheres $\s^3_{\alpha},\,1/3\leq\alpha<1$.

The authors would like to thank A. Ros for his valuable comments about the paper.

\section{Berger spheres}
Let $\s^3=\{(z,w)\in\c^2\,|\,|z|^2+|w|^2=1\}$ the unit sphere, $g$ the standard metric of constant curvature $1$ on $\s^3$
and $V$ the vector field on the sphere given by $$V_{(z,w)}=(iz,iw),\quad (z,w)\in\s^3.$$
On $\s^3$ there are a $2$-parameter family of Berger metrics, but up to homothecy, this family is reduced to the following 1-parameter family of metrics $\{g_{\alpha},\,\alpha>0\}$,
defined by
\[
g_{\alpha}(X, Y) =   g( X, Y) + (\alpha-1) g( X, V) g( Y, V),\quad \alpha>0.
\]
We note that $g_1=g$  and that $g_{\alpha}(V,V)=\alpha$. Along the paper we will discover the big difference between these metrics depending of the sign of $\alpha-1$. From now on will denote by $\s^3_{\alpha}=(\s^3,g_{\alpha})$.

Some important and well-known properties of these Riemannian $3$-manifolds, which will be taking into account along the paper, are the following:

\begin{enumerate}
\item $\s^3_{\alpha}$, $\alpha\not=1$ are homogeneous Riemannian manifolds with isometry group $U(2)$.

\item The Hopf fibration $\Pi:\s^3_{\alpha}\rightarrow\s^2(1/2)$, defined by $$\Pi(z,w)=(z\bar{w},\frac{|z|^2-|w|^2}{2}),$$ is a circle Riemannian submersion onto the $2$-dimensional sphere of radius $1/2$, with totally geodesic fibers and the unit vertical vector field $\xi=\frac{V}{\sqrt\alpha}$ is a Killing field. The bundle curvature is $\sqrt{\alpha}$.

\item The Ricci curvature $S$ of $\s^3_{\alpha}$ is bounded by
$$
S\geq
\begin{cases}
2\alpha, &\hbox{when}\quad\alpha<1\\
2(2-\alpha), &\hbox{when}\quad \alpha>1.
\end{cases}
$$
\item The scalar curvature of $\s^3_{\alpha}$ is $2(4-\alpha)$.
\end{enumerate}
Now we are going to identify these Berger spheres with nice hypersurfaces of very well-known Riemannian $4$-manifolds. This approach to the Berger spheres will be crucial in the proof of some results in the paper.

Let $\m^4(c)$, $c\not=0$ be the complex projective plane $\c\p^2(c)$ of constant holomorphic sectional curvature $4c$ when $c>0$ and the complex hyperbolic plane $\c\h^2(c)$ of constant holomorphic sectional curvature $4c$ when $c<0$, i.e.
\[
\m^4(c)=\{[(z_0,z_1,z_2)]\,|\, z_j \in \c, \, \frac{c}{|c|}|z_0|^2 + |z_1|^2 + |z_2|^2 = \frac{1}{c}\}.
\]

\begin{proposition}\label{prop:embedding-berger-M4}
Let $F_{\alpha}:\s^3_\alpha\rightarrow \m^4(1-\alpha),\,\alpha\not=1$ be the map given by
\[
F_\alpha(z,w)=[(\sqrt{\frac{\alpha}{|1-\alpha|}},z,w)].
\]
Then $F_{\alpha}$ is an isometric embedding of the Berger sphere $\s^3_\alpha$ into $\m^4(1-\alpha)$ and $$F_{\alpha}(\s^3_\alpha)=\{[(z_0, z_1,z_2)]\in \m^4(1-\alpha)\,|\,|z_0|^2=\alpha/|1-\alpha|\}$$ is the geodesic sphere of $\m^4(1-\alpha)$ of center $[(1/\sqrt{|1-\alpha|},0,0)]$ and radius
\[
r=
\begin{cases}
\frac{\arcsin\sqrt{1-\alpha}}{\sqrt{1-\alpha}}\quad  \text{when}\quad \alpha<1\\  \frac{\arcsinh\sqrt{\alpha-1}}{\sqrt{\alpha-1}}\quad \text{when} \quad \alpha>1.
\end{cases}
\]
\end{proposition}
We omit the proof of Proposition \ref{prop:embedding-berger-M4} because it is straightforward. Note that the Berger spheres $\s^3_\alpha$ with $\alpha<1$ are geodesic spheres of the complex projective plane, whereas the Berger spheres $\s^3_{\alpha}$ with $\alpha>1$ are geodesic spheres of the complex hyperbolic plane.

\vspace{0.2cm}

Now we are going to show that other important homogeneous Riemannian $3$-manifolds can be also view as hypersurfaces of $\c\h^2$.

First Tomter \cite{To} showed that the Heisenberg group $\mathrm{\mathrm{Nil}}_3$ can be isometrically embedded in $\c\h^2$ as a horosphere. Second we are going to show that  $\mathrm{Sl}(2,\r)$ can be embedded in $\c\h^2$. We consider $\mathrm{Sl}(2,\r)$ as $$\mathrm{Sl}(2,\r)\equiv\{(z,w)\in\c^2\,|\,|z|^2-|w|^2=1\}$$ and the trivialization of its tangent bundle given by the vector fields $V_{(z,w)}=(iz,iw),\,E^1_{(z,w)}=(\bar{w},\bar{z}),\,E^2_{(z,w)}=(i\bar{w},i\bar{z})$. Then we can define  a $1$-parameter family of metrics $\{g_\beta\,|\,\,\beta>0\}$ on $\mathrm{Sl}(2,\r)$  by
\begin{eqnarray*}
g_\beta(E^i,E^j)=\delta_{ij},\quad g_\beta(V,V)=\beta,\\ g_\beta(V,E^i)=0,\quad i=1,2.
\end{eqnarray*}
In the next result we described, without proof again, the more relevant properties of these Riemannian $3$-manifolds.
\begin{proposition}~ \label{prop:properties-berger-spheres}
\begin{enumerate}
 \item $(\mathrm{Sl}(2,\r),g_\beta)$ is a Riemannian homogeneous $3$-manifold with isometry group $U^1(2)$.
 \item The Hopf fibration $\Pi:\mathrm{Sl}(2,\r)\rightarrow\h^2(-4)$, defined by $$\Pi(z,w)=(z\bar{w},\frac{|z|^2+|w|^2}{2}),$$ is a circle Riemannian submersion from $(\mathrm{Sl}(2,\r),g_\beta)$ onto the $2$-dimensional hyperbolic plane of constant curvature $-4$, with totally geodesic fibers and the unit vertical vector field $\xi=\frac{V}{\sqrt{\beta}}$ is a Killing field. The bundle curvature is $\sqrt\beta$.
     \item The map
\begin{eqnarray*}
G_\beta:(\mathrm{Sl}(2,\r),g_\beta)&\longrightarrow&\c\h^2(-(1+\beta))\\
(z,w)&\mapsto&[(z,w ,\sqrt{\beta/(1+\beta)})],
\end{eqnarray*}
is an isometric embedding with $$G_\beta(\mathrm{Sl}(2,\r))=
M_\beta:=\{[(z_0,z_1,z_2)]\in\c\h^2(-(1+\beta))\,|\,|z_2|=\sqrt{\beta/(1+\beta)}\}.$$
\end{enumerate}
\end{proposition}

In general an oriented hypersurface of $\m^4(c)$  is called {\it pseudo-umbilical} if the shape operator associated to a unit normal vector field $\eta$ has two constant principal curvatures, $\lambda$ and $\mu$ of multiplicities $2$ and $1$ respectively and $J\eta$ is an eigenvector of $\mu$, being $J$ the complex structure of $\m^4(c)$.  Montiel and Takagi classified the pseudo-umbilical hypersurfaces obtaining the following result:

\begin{theorem}[\cite{M, T}] \label{thm:pseudo-umbilical-hypersurfaces-M4}
The geodesic spheres, the horosphere and the hypersurfaces $\{M_\beta\,|\,\beta>0\}$ are the unique pseudo-umbilical hypersurfaces of $\m^4(c)$. Moreover in all the cases the Killing field $\xi$ on the hypersurface is given by $\xi=J\eta$, where $\eta$ is a unit normal vector field.
\end{theorem}

In the case of the Berger spheres, the principal curvatures $\lambda$ and $\mu$ of the isometric embedding $F_\alpha:\s^3_\alpha\rightarrow\m^4(1-\alpha)$ are given by $\lambda=\sqrt{\alpha},\,\mu=\frac{2\alpha-1}{\sqrt\alpha}.$ Hence if $\hat{\sigma}$ is the second fundamental form of the immersion $F_{\alpha}$, then it follows that
\begin{equation}\label{eq:second-ff-berger-M4}
\langle\hat{\sigma}(v,w),\eta\rangle=\sqrt\alpha\langle v,w\rangle+\frac{\alpha-1}{\sqrt\alpha}\langle v,\xi\rangle\langle w,\xi\rangle,
\end{equation}
for any vectors $v,w$ tangent to $\s^3_{\alpha}$, where $\langle,\rangle$ denotes the metric of $\m^4(1-\alpha)$, as well as the induced metric $g_{\alpha}$.

To finish this section, we are going to consider the first standard isometric embedding $\Psi:\c\p^2(1-\alpha)\rightarrow\r^8,\,\alpha<1$ of the complex projective plane into the Euclidean space $\r^8$. The geometric properties of this embedding were studied in \cite{R1} and we emphasize the following one which will be use in the proof of Theorem $7$.

If $\bar{\sigma}$ is the second fundamental form of the first isometric embedding of $\c\p^2(1-\alpha)$ in $\r^8$, then
\begin{equation}\label{eq:second-ff-M4-R8}
\begin{split}
&\langle\bar{\sigma}(x,y),\bar{\sigma}(v,w)\rangle= 2(1-\alpha)\langle x,y\rangle\langle v,w\rangle \\
+(1-\alpha) \bigl(&\langle x,w\rangle\langle y,v\rangle +
\langle x,v\rangle\langle y,w\rangle+\langle x,Jw\rangle\langle y,Jv\rangle+\langle x,Jv\rangle\langle y,Jw\rangle \bigr),
\end{split}
\end{equation}
for any vectors $v,w,x,y$ tangent to $\c\p^2$, where $J$ is the complex structure of $\c\p^2$ and $\langle,\rangle$ denotes the Euclidean metric of $\r^8$.
\section{Surfaces in the Berger spheres}
In this section we are going to set out some known properties of constant mean curvature surfaces of $\s^3_{\alpha}$, which will be use along the paper.

First of all, since  $\xi$ is a unit Killing field on $\s^3_{\alpha}$, for any vector field $X$ tangent to $\s^3_{\alpha}$ we have that
\begin{equation}\label{eq:killing-covariant-derivative}
\nabla^{\alpha}_X\xi=\sqrt{\alpha}(X\wedge\xi),
\end{equation}
where $\nabla^{\alpha}$ is the connection on $\s^3_{\alpha}$ and $\wedge$ is the vectorial product in the $3$-manifold $\s^3_{\alpha}$.

Now, let $\Phi:\Sigma\rightarrow \s^3_{\alpha}$ be an immersion of an orientable surface $\Sigma$ and $N$ a unit normal vector field. We define the function $C:\Sigma\rightarrow \r$ by
\[
C=\langle N,\xi\rangle,
\]
where $\langle,\rangle$ is the metric in $\s^3_{\alpha}$ as well as in $\Sigma$. It is clear that $C^2\leq 1$ and that $\{p\in\Sigma\,|\,C^2(p)=1\}=\{p\in\Sigma\,|\,\xi(p)=\pm N(p)\}$ has empty interior because the distribution orthogonal to $\xi$ on $\s^3_{\alpha}$ is not integrable. If $\gamma$ is the $1$-form on $\Sigma$ given by $\gamma(X)=\langle X,\xi\rangle$, then from \eqref{eq:killing-covariant-derivative} it is easy to prove that
\[
d\gamma=-2\sqrt{\alpha}C\,dA,
\]
and hence, if $\Sigma$ is compact, then $\int_{\Sigma}C\,dA=0$.

We can interpret the function $C$ in terms of the immersion $F_{\alpha}\circ\Phi:\Sigma\rightarrow \m^4(1-\alpha)$. In fact, if $\Omega$ is the Kähler $2$-form on $\m^4(1-\alpha)$, then
\[
(F_{\alpha}\circ\Phi)^*(\Omega)=C\,dA.
\]
This means that $C$ is the Kähler function of the immersion $F_{\alpha}\circ\Phi$. Hence, if $\Sigma$ is compact and $\alpha<1$, we have that the degree of $F_{\alpha}\circ\Phi:\Sigma\rightarrow\c\p^2(1-\alpha)$, which is given by
\[
\text{degree}\,(F_\alpha\circ\Phi)=\frac{1}{2\pi}\int_{\Sigma}C\,dA,
\]
is zero.

On the other hand, the Gauss equation of $\Phi$ is given by (see \cite{D})
\begin{equation}\label{eq:gauss-equation}
K=2H^2-\frac{|\sigma|^2}{2}+\alpha+4(1-\alpha)C^2
\end{equation}
where $K$ is the Gauss curvature of $\Sigma$, $H$ is the mean curvature associated to $N$ and $\sigma$ is the second fundamental form of $\Phi$.

\vspace{0.2cm}

Suppose now  that the immersion $\Phi$ has {\em constant mean curvature.}
 We consider on $\Sigma$ the structure of Riemann surface associated to the induced metric and let $z=x+iy$ be a conformal parameter on $\Sigma$. Then, the induced metric is written as $e^{2u}|dz|^2$ and we denote by $\partial_z=(\partial_x-i\partial_y)/2$ and $\partial_{\bar{z}}=(\partial_x+i\partial_y)/2$ the usual operators.

For these surfaces, the Abresch-Rosenberg quadratic differential $\Theta$, defined by
\[
\Theta(z)=\left(\langle\sigma(\partial_z,\partial_z),N\rangle-
\frac{2(1-\alpha)}{H+i\sqrt\alpha}\langle\Phi_z,\xi\rangle^2\right)(dz)^2,
\]
is holomorphic (see \cite{AR},\cite{D}). We denote  $p(z)=\langle \sigma(\partial_z,\partial_z),N\rangle$ and $A(z)=\langle\Phi_z,\xi\rangle$.

\begin{proposition}[\cite{D, FM}] \label{prop:integrability-conditions}
The fundamental data $\{u,C,H,p,A\}$ of a constant mean curvature immersion $\Phi:\Sigma\rightarrow\s^3_{\alpha}$  satisfy the following integrability conditions:
\begin{equation}\label{eq:integrability-conditions}
\left\{
\begin{aligned}
p_{\bar{z}}&=2(1-\alpha)e^{2u}CA,\\
A_{\bar{z}}&=\frac{e^{2u}C}{2}(H+i\sqrt{\alpha}),\\
C_z&=-(H-i\sqrt\alpha)A-2pe^{-2u}\bar{A},\\
|A|^2&=\frac{e^{2u}}{4}(1-C^2).
\end{aligned}
\right.
\end{equation}
Conversely, if $u,C:\Sigma\rightarrow\r$ with $-1\leq C\leq 1$ and $p,A:\Sigma\rightarrow\c$ are functions on a simply connected surface $\Sigma$ satisfying equations \eqref{eq:integrability-conditions}, then there exists a unique, up to congruences, immersion $\Phi:\Sigma\rightarrow\s^3_{\alpha}$ with constant mean curvature $H$ and whose fundamental data are $\{u,C,H,p,A\}$.
\end{proposition}
Finally, using \eqref{eq:gauss-equation}, the Jacobi operator $L:C^{\infty}(\Sigma)\rightarrow C^{\infty}(\Sigma)$ of the second variation of the area defined in section 1 becomes in
\begin{equation}\label{eq:Jacobi-operator}
\begin{aligned}
L=\Delta+|\sigma|^2+2\alpha+4(1-\alpha)(1-C^2)\\
=\Delta-2K+4(H^2+1)+ 4(1-\alpha)C^2,
\end{aligned}
\end{equation}
being $\Delta$ the Laplacian of $\Sigma$. It is interesting to remark that as $\xi$ is a Killing field on $\s^3_{\alpha}$, then $C=\langle\xi,N\rangle$ is a Jacobi function on $\Sigma$, i.e. $LC=0$.

\section{Stability of constant mean curvature spheres}
In this section we are going to study which constant mean curvature spheres of $\s^3_{\alpha}$ are stables. The constant mean curvature spheres of $\s^3_{\alpha}$ are surfaces invariant under a $1$-parameter group of isometries of $U(2)$ and their existence was announced by Abresch in \cite{A}. More recently, Torralbo \cite{Tr} has described and studied them, showing that:
 \begin{quote}
 \emph{For each real number $H\geq 0$ there exists, up to congruences, a unique immersed constant mean curvature sphere $\mathcal{S}_{\alpha}(H)$ in $\s^3_{\alpha}$ with constant mean curvature $H$}.
\end{quote}
So in each Berger sphere $\s^3_{\alpha}$ there are a $1$-parameter family of constant mean curvature spheres parametrized by the mean curvature $H$, with $H\in[0,\infty[$. All the minimal spheres in this family are nothing but a great equator in $\s^3$, that is, up to congruences
\[
\mathcal{S}_{\alpha}(0)=\{(z,w)\in\s^3\,|\,\Im (w)=0\},\quad\forall\,\alpha>0.
\]
Almost all the spheres $\mathcal{S}_{\alpha}(H)$ are embedded surfaces, but in \cite{Tr} it is proved that when $\alpha$ is very close to zero there are constant mean curvature spheres which are not embedded (see figure~\ref{fig:immersed-cmc-sphere}).

\begin{figure}[htbp]
\centering
\includegraphics[width=0.7\textwidth]{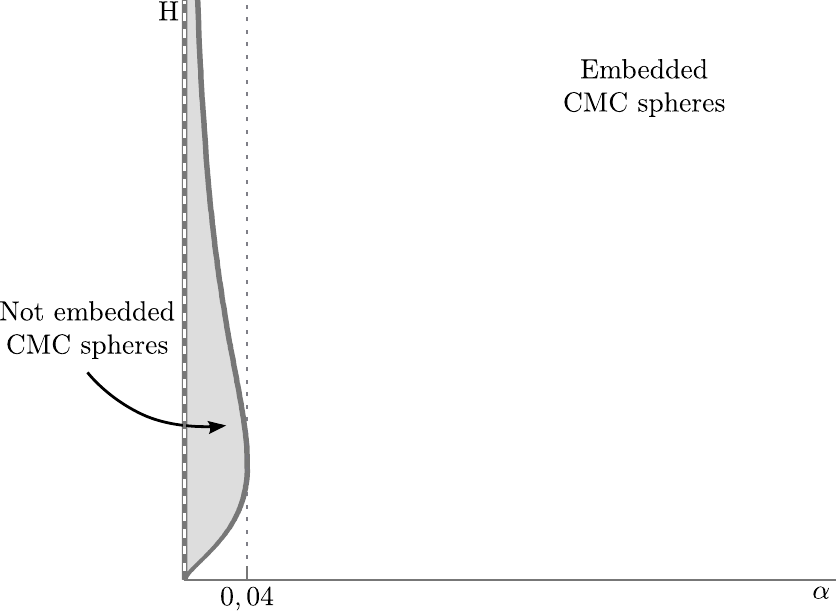}
\caption{Non-embedded region of CMC spheres}\label{fig:immersed-cmc-sphere}
\end{figure}

\begin{theorem} \label{thm:stability-spheres}
There exists $\alpha_0 \in \, ]0,1[$ ($\alpha_0\approx 0.121$) such that:
\begin{enumerate}
    \item for $\alpha \geq \alpha_0$ the spheres $\{\mathcal{S}_\alpha(H)\, | \, H \geq 0\}$ are stable in $\s^3_\alpha$,

    \item for $\alpha<\alpha_0$, there exists $H(\alpha)>0$ such that $\mathcal{S}_\alpha(H)$ is stable in $\s^3_\alpha$ if and only if $H\geq H(\alpha)$. (See figure 2).
\end{enumerate}
\end{theorem}

\begin{proof}
If an orientable compact constant mean curvature surface of $\s^3_{\alpha}$ is stable, then the index of the quadratic form $Q$ associated to the Jacobi operator $L$ is one, because (see Corollary 9.6, [PMR]) the index must be positive and if the index is greater than $1$, we can easily find a smooth function $f$ with $\int_{\Sigma}f\,dA=0$ and satisfying $Q(f)<0$, which contradicts the stability of $\Sigma$. Now, among the constant mean curvature surfaces with index $1$, Koiso in \cite{K} (see also \cite{S}) gave an stability criterium that we are going to use to prove the Theorem.

\begin{theorem}[\cite{K, S}]\label{thm:Koiso-stability-criteria}
Let $\Phi:\Sigma\rightarrow\s^3_{\alpha}$ be a CMC immersion of a compact orientable surface $\Sigma$. Suppose that $\Sigma$ has index $1$ and $\int_{\Sigma}f\,dA=0$ for any Jacobi function $f$, i.e., any function satisfying $Lf=0$. Then there exists a uniquely determined smooth function $v\in\ker{L}^{\perp}$ satisfying $Lv=1$. Moreover $\Sigma$ is stable if and only if $\int_{\Sigma}v\,dA\geq0$.
\end{theorem}
In order to use Koiso's Theorem, we start by proving the following result:

\begin{quote}
{\em The quadratic forms associated to the Jacobi operators of any constant mean curvature sphere $\mathcal{S}_{\alpha}(H)$ are the same. In particular their index is one, their nullity is three and $\int_{\Sigma} f\,dA=0$ for any Jacobi function. }
\end{quote}

Let $\Phi:\Sigma\rightarrow\s^3_{\alpha}$ be a constant mean curvature immersion of a sphere. Then the Abresch-Rosenberg holomorphic differential $\Theta$ vanishes identically and so from \eqref{eq:integrability-conditions} we obtain that
\begin{equation}\label{eq:gradient-laplacian-C}
\begin{split}
C_z &=-\frac{A}{H+i\sqrt\alpha}(H^2+1-(1-\alpha)C^2),\\
C_{z\bar{z}} &=-\frac{e^{2u}C}{2(H^2+\alpha)}(H^2+1-(1-\alpha)C^2)^2.
\end{split}
\end{equation}
As $H^2+1-(1-\alpha)C^2>0$, the only critical points of $C$ are those where $A$ vanishes, i.e., taking into account \eqref{eq:integrability-conditions}, those with $C^2=1$. Moreover, using \eqref{eq:integrability-conditions} again, the determinant of the Hessian of $C$ in a critical point is $(H^2+\alpha)^2>0$. So $C$ is a Morse function on $\Sigma$ and so it has only two critical points $p,q$ which are the absolute maximum and minimum of $C$.

Using one more time \eqref{eq:integrability-conditions} it is easy to check that
$\log\sqrt{\frac{1+C}{1-C}}$ is a harmonic function with singularities at $p$, $q$ and without critical points. Hence there exists a global conformal parameter $z$ on the sphere $\Sigma$ such that $(\log\sqrt{\frac{1+C}{1-C}})(z)=\log|z|$, and so the function $C$ of any constant mean curvature immersion of $\Sigma$ is given by $$C(z)=\frac{|z|^2-1}{|z|^2+1},\quad z\in\bar{\c}.$$

Conversely, let $C:\bar{\c}\rightarrow\r$ be the function $C(z)=\frac{|z|^2-1}{|z|^2+1}$, $H$ a nonnegative real number and $\s^3_\alpha$ a Berger sphere. We define the functions $A,p:\bar{\c}\rightarrow\c$ and $u:\bar{\c}\rightarrow\r$ by
\[
\begin{split}
A(z)&=-\frac{2(H+i\sqrt\alpha)\bar{z}}{4(1-\alpha)|z|^2+(H^2+\alpha)(|z|^2+1)^2},\\
p(z)&=\frac{2(1-\alpha)}{H+i\sqrt\alpha}A^2(z),\\
e^{2u(z)}&=\frac{4(|z|^2+1)^2(H^2+\alpha)}{(4(1-\alpha)|z|^2+(H^2+\alpha)(|z|^2+1)^2)^2}.
\end{split}
\]
Then it is easy to check that these functions satisfy the equations \eqref{eq:integrability-conditions} and so Proposition \ref{prop:integrability-conditions} says that there exists a conformal immersion $\Phi:\bar{\c}\rightarrow\s^3_{\alpha}$ with constant mean curvature $H$. We note that for any $H$ and for any Berger sphere $\s^3_{\alpha}$ the immersions $\Phi's$ have associated the same function $C$.

Also, from \eqref{eq:integrability-conditions} and \eqref{eq:Jacobi-operator} it is straightforward to check that the Jacobi operator $L$ of the immersion $\Phi$ is given by $$L=\Delta+q,\quad\quad q(z)=\frac{8e^{-2u(z)}}{(|z|^2+1)^2}.$$
Suppose that $\hat{\Phi}:\bar{\c}\rightarrow\s^3_{\hat{\alpha}}$ is an immersion with constant mean curvature $\hat{H}$. Then its Jacobi operator is
$\hat{L}=\hat{\Delta}+\hat{q}$, with $\hat{q}(z)=\frac{8e^{-2\hat{u}(z)}}{(|z|^2+1)^2}$. As $\hat{q}e^{2\hat{u}}=qe^{2u}$, the quadratics forms $\hat{Q}$ and $Q$ of $\hat{\Phi}$ and $\Phi$ satisfy
\[
\begin{split}
\hat{Q}(f) &= \int_{\bar{\c}}-f(\hat{\Delta}f+\hat{q}f)e^{2\hat{u}}dz= \int_{\bar{\c}}-f(e^{2\hat{u}}\hat{\Delta}\,f+qe^{2u}f)\,dz\\
&=\int_{\bar{\c}}-f(e^{2u}\Delta\,f+qe^{2u}f)\,dz=Q(f),
\end{split}
\]
for any smooth function $f:\bar{\c}\rightarrow\r$.
Hence all the constant mean curvature spheres have the same quadratic form of the second variation. In particular, all have the same index and nullity. Also it is clear that the above property keeps for constant mean curvature spheres of the round sphere $\s^3_{1}$ whose index are $1$ and their nullity are $3$.

To finish the proof of this first step, as the Berger spheres have $4$ linearly independent Killing fields and the constant mean curvature spheres are invariant under a $1$-parameter group of isometries, then all the Jacobi functions come from Killing fields of the ambient space, i.e., if $Lf=0$ then there exists a Killing field $V$ on the ambient space such that $f=\langle V,N\rangle$. Now it is clear that if $V^{\top}$ denotes the tangential component of $V$, then $\text{div}\,V^{\top}=2Hf$, and hence $\int_{\Sigma}f\,dA=0$ if the sphere is not minimal. In the minimal case is easy to prove the same property. Hence the proof of our claim has finished.

Now from Koiso's result, there exists a unique function $v\in(\ker\,L)^{\perp}$ with $Lv=1$. It is clear that another function $f$ with $Lf=1$ is given by $f=v+f_0$ with $f_0\in\ker\,L$ and hence $\int_{\Sigma}f\,dA=\int_{\Sigma}v\,dA$. So for the stability criterium we can use any function $f$ with $Lf=1$.

In order to get a solution of $Lf=1$, it is convenient to reparametrize the spheres by $e^w=z$. Then if $w=x+iy$, the function $C$ becomes in $C(x,y)=\tanh x$ and the induced metric in $e^{2v}|dw|^2$, where $v(x,y)=v(x)$. A simple computation shows that the equation $Lf=1$ becomes in
\[
f''(x)+\frac{2}{\cosh^2x}f(x)=\frac{(H^2+\alpha)\cosh^2x}{\left((H^2+\alpha)\cosh^2x+1-\alpha)\right)^2}.
\]
It is straightforward to check that if $h(x)=\frac{\sqrt{|1-\alpha|}}{\sqrt{H^2+1}}\tanh\,x$, then
\[
	f(x) =
	 \begin{cases}
	  \displaystyle\frac{1}{2(H^2+1)}\left[1-h(x)\arctanh( h(x))\right], & \text{when } \alpha < 1,\\

	\displaystyle\frac{1}{2(H^2+1)}\left[1+h(x)\arctan( h(x))\right] & \text{when } \alpha > 1,
	 \end{cases}
	\]
is a solution of $Lf=1$.

Now
\[
\int_{\bar{\c}}f\,dA=
\begin{cases}
    \frac{\pi}{2(H^2 + 1)^2} \left(3 + \frac{(H^2 + 3\alpha - 2)}{\sqrt{H^2+1}\sqrt{1-\alpha}} \arctanh \frac{\sqrt{1-\alpha}}{\sqrt{H^2+1}} \right), & \text{ when } \alpha < 1 \\
    \frac{\pi}{2(H^2 + 1)^2} \left(3 + \frac{(H^2 + 3\alpha - 2)}{\sqrt{H^2+1}\sqrt{\alpha-1}} \arctan \frac{\sqrt{\alpha-1}}{\sqrt{H^2+1}} \right), & \text{ when } \alpha > 1. \\
\end{cases}
\]
When $\alpha > 1$, the above integral is always positive and hence the corresponding constant mean curvature spheres are stable. On the other hand when $\alpha < 1$, the above integral is positive in the region showed in the figure~\ref{fig:stable-CMC-spheres}, where the painting curve has implicit equation $\int f\, dA = 0$, i.e.
\[
    3\sqrt{H^2 + 1}\sqrt{1-\alpha} + (H^2 + 3\alpha - 2)\arctanh\left( \frac{\sqrt{1-\alpha}}{\sqrt{H^2 + 1}} \right) = 0.
\]
The special value $\alpha_0$ is the solution of the above equation when $H = 0$, i.e., $\arctanh(\sqrt{1-\alpha_0}) = \dfrac{3\sqrt{1-\alpha_0}}{2 - 3\alpha_0}$. Using again Koiso's theorem we finish the proof of the Theorem.
\end{proof}
\begin{figure}[htbp]
\centering
\includegraphics[width=0.7\textwidth]{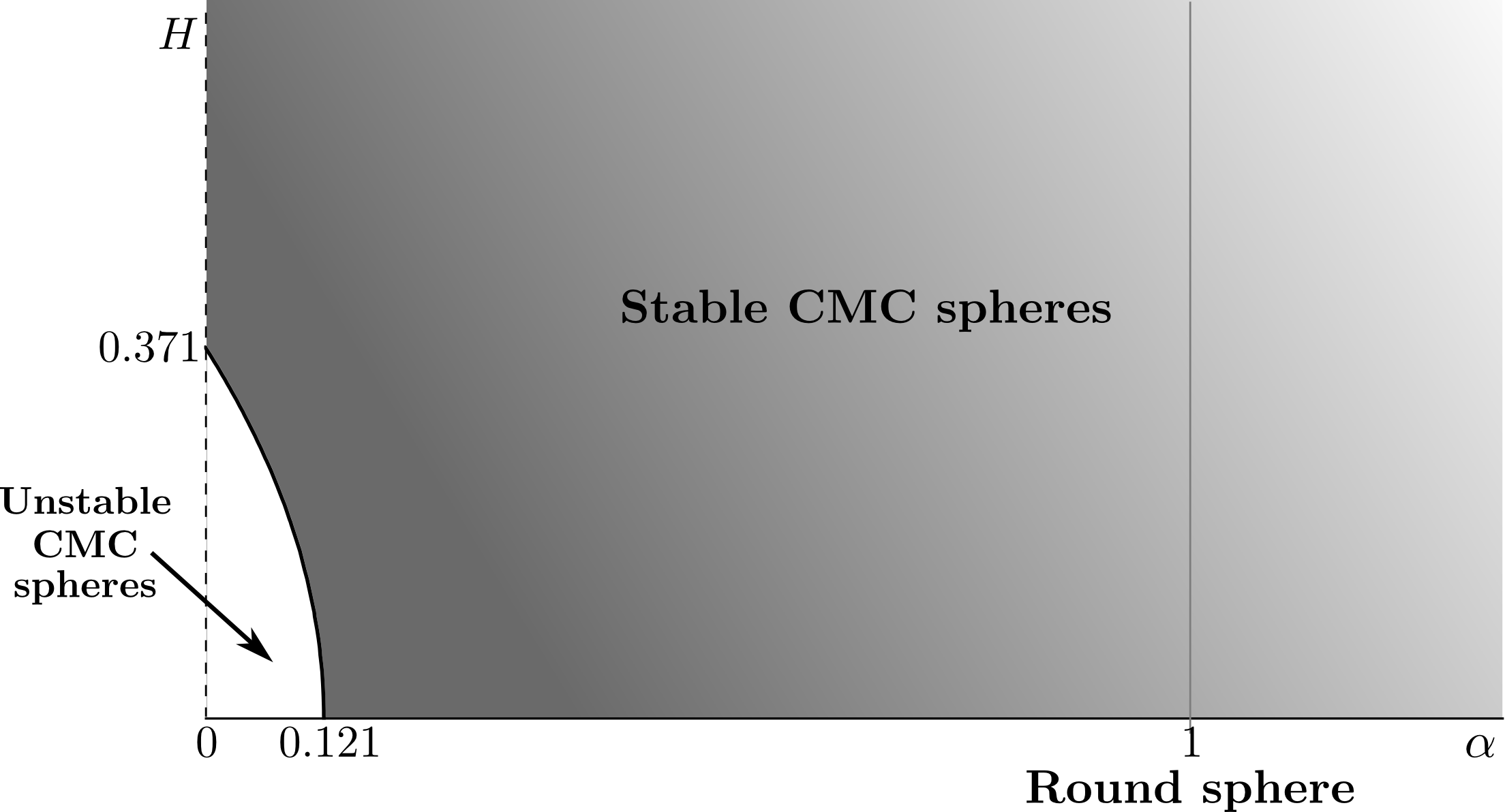}
\caption{Stable CMC spheres.\label{fig:stable-CMC-spheres}}
\end{figure}

The proof of Theorem $2$ can be adapted to study the stability of the constant mean curvature spheres of the Heisenberg group and the homogeneous Riemannian manifolds $(\mathrm{Sl}(2,\r),g_\beta)$. In fact following the same reasoning it is easy to see that these constant mean curvature spheres have also the same quadratic form $Q$ than the corresponding spheres of $\s^3_{\alpha}$, and hence their index are one and their nullity are three. Now defining a function $f$ with $Lf=1$ like in the case of the Berger spheres, it can be proved that $\int_{\Sigma}f\geq 0$ for any constant mean curvature sphere of $\mathrm{Nil}_3$ or $(\mathrm{Sl}(2,\r),g_\beta)$. Hence we obtain the following result:
\begin{theorem}\label{thm:stability-spheres-Nil3-Sl}
For each $H>0$ the unique embedded constant mean curvature sphere in the Heisenberg group $\mathrm{Nil}_3$ with constant mean curvature $H$ is stable.

For each $H> 1$ the unique immersed constant mean curvature sphere in the homogeneous Riemannian $3$-manifold $(\mathrm{Sl}(2,\r),g_\beta)$ with constant mean curvature $H$ is stable.
\end{theorem}
Torralbo \cite{Tr} proved that when $\beta$ is very close to $0$ or $H$ is very close to $1$ there are constant mean curvature spheres which are not embedded.

\section{Examples of stable constant mean curvature tori}
In \cite{TU} the authors classified the compact {\it flat} surfaces in $\s^3_{\alpha}$, proving that they are Hopf tori, i.e., inverse images of closed curved of $\s^2(1/2)$ by the Hopf fibration $\Pi:\s^3_{\alpha}\rightarrow\s^2(1/2)$ (see section 2). Since the Killing vector field $\xi$ is tangent to a Hopf torus, these surfaces have $C=0$. Now it is easy to check that such Hopf torus has constant mean curvature $H$ if and only if the closed curve of $\s^2(1/2)$ has constant curvature $2H$. As, up to congruences of $\s^2(1/2)$, these curves are intersection of $\s^2(1/2)$ with horizontal planes of $\r^3$, we get the following result.

\begin{proposition}\label{prop:CMC-flat-torus}
For each $H\geq 0$ there exists, up to congruences, a unique constant mean curvature embedded flat torus $\mathcal{ T}_{\alpha}(H)$ in $\s^3_\alpha$ with constant mean curvature $H$. Such torus is defined by
\[
    \mathcal{T}_{\alpha}(H)=\{(z, w) \in \s^3 \,|\, |z|^2 = r_1^2,\,|w|^2=r_2^2\}=\s^1(r_1)\times\s^1(r_2)
\]
with
\[
\quad r_1^2 =\frac{1}{2} + \frac{H}{2\sqrt{1+H^2}}\quad \text{and}\quad r_2^2 =\frac{1}{2} - \frac{H}{2\sqrt{1+H^2}}.
\]
\end{proposition}

\begin{remark}
We note that the torus $\s^1(r_1)\times\s^1(r_2)$ with $\quad r_1^2 =\frac{1}{2} + \frac{H}{2\sqrt{1+H^2}}$ has constant mean curvature $H$ with respect to any Berger sphere $\s^3_{\alpha},\,\alpha>0$. In particular, for $H=0$ all the tori $\mathcal{T}_{\alpha}(0)$ are the Clifford torus in $\s^3$,
\[
\mathcal{ T}_{\alpha}(0)=\{(z,w)\in\s^3\,|\,|z|^2=|w|^2=\frac{1}{2}\}.
\]
\end{remark}

Now we are going to determine which of these flat tori are stable as constant mean curvature surfaces. As $K=0$ and $C=0$, from \eqref{eq:Jacobi-operator} it follows that the Jacobi operator of the torus $\mathcal{T}_{\alpha}(H)$ is given by
\[
L=\Delta+4(H^2+1).
\]
Hence we need to compute the first non-null eigenvalue of the Laplacian $\Delta$ of the torus $\mathcal{T}_{\alpha}(H)$. To do that, let $\Phi:\r^2\rightarrow\s^3_{\alpha}$ the immersion $\Phi(t,s)=(r_1e^{it},r_2e^{is})$. Then $\Phi(\r^2)=\mathcal{T}_{\alpha}(H)$, where the relation between $r_i$ and $H$ is given in Proposition \ref{prop:CMC-flat-torus}. The induced metric is given by $g=(g_{ij})$ with $g_{ii}=r_i^2(1-(1-\alpha)r_i^2)$ for $i=1,2$ and $g_{12}=-r_1^2r_2^2(1-\alpha)$. Therefore intrinsically the torus $\mathcal{T}_{\alpha}(H)$ is given by $T=\r^2/\Lambda$, $\Lambda$ being the lattice in $\r^2$ generated by the vectors $2\pi v_1$ and $2\pi v_2$ where
\[
v_1=(r_1\sqrt{1-(1-\alpha)r_1^2},0),\quad v_2=\frac{r_2}{\sqrt{1-(1-\alpha)r_1^2}}(-r_1r_2(1-\alpha),\sqrt{\alpha}).
\]
Now, the dual lattice is generated by
\[
v_1^*=\frac{1}{\sqrt{1-(1-\alpha)r_1^2}}(\frac{1}{r_1},\frac{r_2(1-\alpha)}{\sqrt\alpha}),\quad v_2^*=(0,\frac{\sqrt{1-(1-\alpha)r_1^2}}{r_2\sqrt\alpha}),
\]
and hence the spectrum of the Laplacian of $\mathcal{T}_\alpha(H)$ is given by $\{|mv_1^*+nv_2^*|^2\,|\,m,n\in\z\}$. Taking into account that $r_1\geq r_2$, it is not difficult to check that the first non-null eigenvalue $\lambda_1$ of the Laplacian of $\mathcal{T}_\alpha(H)$ is given by:
\[
\lambda_1=
\begin{cases}\frac{2\sqrt{H^2+1}}{H+\sqrt{H^2+1}}+\frac{1-\alpha}{\alpha}\quad\text{when}
\begin{cases}
\alpha>1/3\quad\text{or}\\
\alpha\leq\,1/3\quad\text{and}\,\quad  H>\frac{1-3\alpha}{2\sqrt{\alpha(1-2\alpha)}}
\end{cases}\\
4(H^2+1)\quad \text{when}\quad \alpha\leq\,1/3\quad\text{and}\quad H\leq\frac{1-3\alpha}{2\sqrt{\alpha(1-2\alpha)}}.
\end{cases}
\]
So, the expression of the Jacobi operator of $\mathcal{T}_\alpha(H)$ gives the following result.

\begin{proposition}~ \label{prop:stability-flat-tori}
\begin{enumerate}
	\item For each $\alpha>1/3$ the constant mean curvature tori $\{\mathcal{T}_\alpha(H)\,|\,H\geq 0\}$ are unstable in $\s^3_\alpha$.
	\item For each $\alpha$ such that $\alpha\leq 1/3$ the constant mean curvature torus ${ \mathcal{T}}_\alpha(H)$ is stable if and only if $H\leq\frac{1-3\alpha}{2\sqrt{\alpha(1-2\alpha)}}$. (See figure 3).
\end{enumerate}
\end{proposition}

\begin{figure}[h]
\centering
\includegraphics[width=0.7\textwidth]{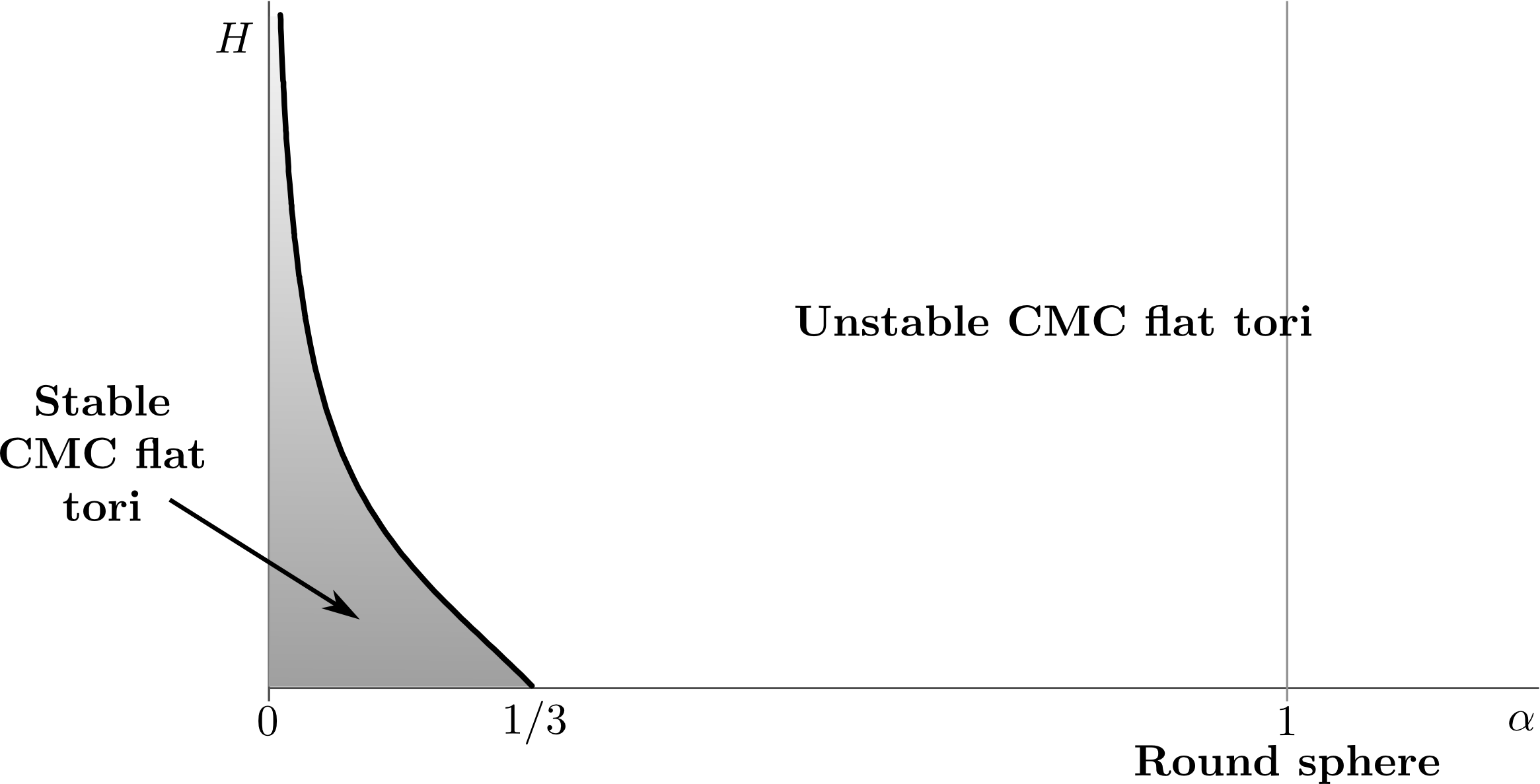}
\caption{Stable CMC flat tori.\label{fig:stable-flat-tori}}
\end{figure}

\begin{remark}
For each Berger sphere $\s^3_\alpha$ with $\alpha\leq 1/3$, the Clifford torus $\mathcal{T}_\alpha(0)$ is stable and in $\s^3_{1/3}$ the Clifford torus is the only stable constant mean curvature flat torus.
\end{remark}

\section{Stability of compact surfaces}
In this section we are going to study compact stable constant mean curvature surfaces of the Berger spheres $\s^3_\alpha$. We start showing the known results. First (see section 2), a Berger sphere has non-negative Ricci curvature if and only if $\alpha\leq 2$. Ros \cite{R2} bounded the genus of a compact stable constant mean curvature surface of a $3$-dimensional Riemannian manifold with non-negative Ricci curvature. This fact joint with the classical result of Barbosa, DoCarmo and Eschenburg become in the following result:

\begin{theorem}[\cite{BCE},\cite{R2}]~ \label{thm:compact-stable-CMC-surfaces-genus<=3}
\begin{enumerate}
\item The only orientable compact stable constant mean curvature surfaces of the round sphere $\s^3$ ($\alpha=1)$ are the umbilical spheres.
\item If $\Sigma$ is an orientable compact stable constant mean curvature surface of the Berger sphere $\s^3_{\alpha}$ with $\alpha\leq 2$ then the genus $g$ of $\Sigma$ is $g\leq 3$.
\end{enumerate}
\end{theorem}
Now we prove two stability Theorems and the more important ingredient to do that will be to consider the constant mean curvature surfaces of $\s^3_\alpha$ as surfaces in $\m^4(1-\alpha)$ or even in $\r^8$ when $\alpha<1$.

\begin{theorem}\label{thm:compact-stable-CMC-surfaces}
Let $\Phi:\Sigma\rightarrow\s^3_\alpha$ be a constant mean curvature immersion of an orientable compact surface $\Sigma$ with $\alpha_1\leq\alpha\leq 4/3,\,\alpha\not=1$, where $\alpha_1\approx 0.217$. If $\Sigma$ is stable, then $\Sigma$ is either a sphere (in this case all the spheres are stable) or $\Sigma$ is an embedded torus satisfying
\[
\begin{cases}
 (H^2+1)\hbox{Area}\,(\Sigma)<4\pi\quad \text{when}\quad \alpha_0\leq\alpha<1,\\
 H^2\hbox{Area}\,(\Sigma)<4\pi,\quad\text{ when}\quad 1<\alpha\leq 4/3.
 \end{cases}
 \]
\end{theorem}
\begin{remark}
The above result is also true when $\Sigma$ is a {\it complete} orientable surface, because $\s^3_{\alpha},\,\alpha<2$ has positive Ricci curvature and then the compactness of $\Sigma$ comes from Theorem 1 in \cite{RR}.
\end{remark}
\begin{proof}
We use a known argument, coming from the Brill-Noether theory, to get test functions in order to study the stability. Using this theory, we can get a nonconstant meromorphic map $\phi:\Sigma\rightarrow \s^2$ of degree $d\leq 1+[(g+1)/2]$, where $[.]$ stands for integer part and $g$ is the genus of $\Sigma$. Although the mean value of $\phi$ is not necessarily zero, using an argument of Yang and Yau \cite{YY}, we can find a conformal transformation $F:\s^2\rightarrow\s^2$ such that $\int_{\Sigma}(F\circ\phi)\,dA=0$, and so the stability of $\Sigma$ implies that  $0\leq Q(F\circ\phi)$. From \eqref{eq:def-stability}, the Gauss equation \eqref{eq:gauss-equation} and as $|F\circ\phi|^2=1$ it follows that
\[
\int_{\Sigma}|\nabla(F\circ\phi)|^2\,dA\geq 2\int_{\Sigma}(2H^2-K+2+2(1-\alpha)C^2)\,dA.
\]
But $\int_{\Sigma}|\nabla(F\circ\phi)|^2\,dA=8\pi\text{degree}\,(F\circ\phi)=8\pi\text{degree}\,(\phi)\leq 8\pi(1+[(g+1)/2]$. Hence using the Gauss-Bonnet Theorem, the above inequality becomes in
\begin{equation}\label{eq:integral-formula}
2\pi(2-g+[(g+1)/2])\geq \int_{\Sigma}(H^2+1+(1-\alpha)C^2)\,dA.
\end{equation}
First, if $\s^3_{\alpha}$ has nonnegative Ricci curvature, i.e., $\alpha\leq 2$, then as $C^2\leq 1$ it follows that $(1-\alpha)C^2\geq -1$ and hence \eqref{eq:integral-formula} becomes in
\[
2\pi(2-g+[(g+1)/2])> \int_{\Sigma}H^2\,dA\geq 0,
\]
because the function $C^2$ cannot be identically $1$ on $\Sigma$. The above inequality implies that the genus $g$ of $\Sigma$ is $g\leq 3$. This proves the result of Ros \cite{R2} in the case of the Berger spheres (Theorem~\ref{thm:compact-stable-CMC-surfaces-genus<=3},(2)).

Now, in order to prove the Theorem, we use the following result of Montiel and Urbano \cite{MU}:
\begin{quote}
{\em Let $\tilde{\Phi}:\Sigma\rightarrow\m^4(1-\alpha)$ be an immersion of a compact surface and $\mu$ its maximum multiplicity. If $\tilde{H}$ is the mean curvature vector of $\tilde{\Phi}$ and $C$ the Kähler function, then
\[
\int_{\Sigma}(|\tilde{H}|^2+2(1-\alpha)+2|(1-\alpha)C|)\,dA\geq4\pi\mu.
\]
Moreover the equality holds only for certain surfaces of genus zero.}
\end{quote}
In \cite{MU}, the above result was proved only when the ambient space was the complex projective plane, but, as it was indicated in \cite{MU}, slight modifications of the proof would prove the result also for the complex hyperbolic plane.

We consider the Montiel-Urbano result for the immersion $\tilde{\Phi}=F_\alpha\circ\Phi:\Sigma\rightarrow \m^4(1-\alpha)$. Then, from \eqref{eq:second-ff-berger-M4} it is easy to see that
\[
|\tilde{H}|^2=H^2+\frac{(3\alpha-1+(1-\alpha)C^2)^2}{4\alpha},
\]
and so \eqref{eq:integral-formula} joint with the inequality of Montiel and Urbano becomes in
\begin{equation}\label{eq:integral-formula-MU}
2\pi(2-g-\mu+[(g+1)/2])\geq \int_{\Sigma}(\frac{H^2}{2}+\frac{F}{8\alpha})\,dA,
\end{equation}
where $F=-(C^4+2C^2-8\epsilon|C|+1)\alpha^2+2(C^4-4\epsilon|C|+3)\alpha-(1-C^2)^2$, being $\epsilon$ the sign of $1-\alpha$.

To determine for which values of $\alpha$ the function $F$ is nonnegative and as $0\leq|C|\leq 1$ and $C^2=|C|^2$, we define for each $t\in[0,1]$ the second degree polynomial $P_t:\r\rightarrow\r$ by
\[
P_t(\alpha)=A(t)\alpha^2+B(t)\alpha+C(t),
\]
where $A(t)=-(t^4+2t^2-8\epsilon t+1)$, $B(t)=2(t^4-4\epsilon t+3)$ and $C(t)=-(1-t^2)^2$.

Two important properties of this second degree polynomial are:
\begin{enumerate}
\item $P_t(1)=4,\,\forall t\in[0,1]$,
\item the discriminant of $P_t$ is given by $32(t-\epsilon)^2(1+t^2)$ and so for each $t\in[0,1]$ $P_t$ have two roots except when $t=1$ and $\epsilon=1$.
\end{enumerate}

 First we consider the case $\epsilon=-1$, i.e., $\alpha>1$. In this case $A(t)<0,\forall t\in[0,1]$ and the root $\alpha(t)$ of the above polynomial given by
\[
\alpha(t)=\frac{t^4+4t+3+2(1+t)\sqrt{2(1+t^2)}}{t^4+2t^2+8t+1}
\]
is always greater than $1$. Hence for each $t\in[0,1]$, $P_t(\alpha)\geq 0$ if $1<\alpha\leq\alpha(t)$. Hence if $\alpha\leq\min_{t\in[0,1]}\alpha(t)=\frac{4}{3}$ then $P_t(\alpha)\geq 0,\,\forall t\in[0,1]$. This proves that $F\geq 0$ if $1<\alpha\leq 4/3$.

Second we consider the case $\epsilon=1$, i.e., $\alpha<1$. In this case we have that there exists a unique $t_0\in]0,1[$ ($t_0\approx 0,1292$) with $A(t_0)=0$. Moreover if $0\leq t<t_0$ (respectively $t_0< t\leq1$) then $A(t)<0$ (respectively $A(t)>0$). Now the root $\alpha(t),\,t\not=t_0$ of the above polynomial given by
\[
\alpha(t)=\frac{t^4-4t+3-2(1-t)\sqrt{2(1+t^2)}}{t^4+2t^2-8t+1}
\]
satisfies $0<\alpha(t)<1$. Hence for each $0\leq t< t_0$, and as in this case $A(t)<0$, it follows that $P_t(\alpha)\geq 0$ if $\alpha(t)\leq\alpha<1$. Also, if $t_0< t\leq 1$ it is easy to see that the other root of the polynomial
\[
\beta(t)=\frac{t^4-4t+3+2(1-t)\sqrt{2(1+t^2)}}{t^4+2t^2-8t+1}
\]
satisfies $\beta(t)\leq\alpha(t)$ and hence, as in this case $A(t)>0$, it follows also that $P_t(\alpha)\geq 0$ if $\alpha(t)\leq\alpha<1$.

Hence $P_t(\alpha)\geq 0,\,\forall t\in[0,1]$ if $\alpha\geq\max_{t\in[0,1]}\alpha(t)$. It is clear that the function $\alpha(t)$ with $t\in[0,1]$ has only a critical point which is a maximum. We define $\alpha_1=\max_{t\in[0,1]}\alpha(t)\approx 0.217 $. Hence $P_t(\alpha)\geq 0,\,\forall t\in[0,1]$ when $\alpha_1\leq\alpha<1$. This proves that $F\geq 0$ if $\alpha_1\leq\alpha<1$.

 Finally we have proved that $F\geq 0$ if $\alpha_1\leq \alpha\leq 4/3$. Hence from \eqref{eq:integral-formula-MU} we obtain that $2-g-\mu+[(g+1)/2]\geq 0$. This inequality implies that besides the surfaces of genus zero only the following possibilities can happen: $g=1,\mu\leq 2$ or $g=2,3,\mu=1$. But except for the case $g=1$ and $\mu=1$, for the other three possibilities ($g=1,\mu=2$ and $g=2,3, \mu=1$) the equality is attained in the above inequality. In particular the equality is also attained in the Montiel-Urbano inequality, but this is impossible because the genus is not zero. Hence if $\Sigma$ is not a sphere, $\Sigma$ must be an embedded torus. Finally the estimation of the areas comes directly from \eqref{eq:integral-formula} and \eqref{eq:integral-formula-MU}.
\end{proof}

\begin{remark} \label{remark:new-proof-BCE}
The idea developed in the proof of Theorem $6$ can be used to give a new proof of the Barbosa, DoCarmo and Eschenburg result (Theorem~\ref{thm:compact-stable-CMC-surfaces-genus<=3},1) \cite{BCE}). In fact, if $\Phi:\Sigma\rightarrow\s^3$ is a stable constant mean curvature immersion of an orientable compact surface $\Sigma$, then following the proof of Theorem $6$ we obtain in this case
\[
2\pi(2-g+[(g+1)/2])\geq \int_{\Sigma}(H^2+1)\,dA.
 \]
But if $\tilde{H}$ is the mean curvature vector of the immersion $\Phi:\Sigma\rightarrow\s^3\subset\r^4$, then $|\tilde{H}|^2=H^2+1$ and so, using a result of Li and Yau \cite{LY},
\[
\int_{\Sigma}(H^2+1)\,dA=\int_{\Sigma}|\tilde{H}|^2dA\geq 4\pi\mu,
\]
being $\mu$ the maximum multiplicity of the immersion. So the above equation becomes in
\[
2-g-2\mu+[(g+1)/2]\geq0.
\]
So $g=0,1$ and $\mu=1$. In both cases the equality holds in the last inequality and hence $\int_{\Sigma}|\tilde{H}|^2dA= 4\pi$. This implies that the surface is an umbilical sphere.
\end{remark}

\begin{theorem}\label{thm:compact-CMC-stable-1/3<alpha<1}
Let $\Phi:\Sigma\rightarrow\s^3_\alpha$ be a constant mean curvature immersion of an orientable compact surface $\Sigma$ with $1/3\leq\alpha<1$. Then $\Sigma$ is stable if and only if $\Sigma$ is either a sphere (in this case all the spheres are stable) or $\Sigma$ is the Clifford torus $T_{1/3} (0)$ in $\s^3_{1/3}$.
\end{theorem}

\begin{remark}
The argument of Remark 3 shows again that the above result is also true when $\Sigma$ is a complete orientable surface.

The number $\alpha_1$ appearing in Theorem $6$ satisfies $\alpha_1<1/3$. So to prove Theorem $7$ we can assume that our surface is an embedded torus. However for completeness we will prove the result without this assumption.
\end{remark}

\begin{proof}
Theorem \ref{thm:stability-spheres} and Proposition \ref{prop:stability-flat-tori} say that any constant mean curvature sphere is stable and that the Clifford torus $T_{1/3} (0)$ is also stable.

Suppose now that $\Sigma$ is stable and that the genus $g$ of $\Sigma$ is $g\geq 1$.

As $\alpha<1$, we consider (see section 2) the embedding $F_{\alpha}:\s^3_{\alpha}\rightarrow\c\p^2(1-\alpha)$ and the first standard embedding $\Psi:\c\p^2(1-\alpha)\rightarrow\r^8$. So we will consider our surface immersed in $\r^8$, by $\hat{\Phi}=\Psi\circ F_{\alpha}\circ\Phi:\Sigma\longrightarrow\r^8$.

   From the Hodge theory we now that the linear space of harmonic vector fields on $\Sigma$ has dimension $2g$. So let $X$ be a harmonic vector field of $\Sigma$. Then it is well-known that
   \begin{equation}\label{eq:harmonic-vector-field}
   \text{div}\,X=0\quad \text{and}\quad \Delta^{\Sigma}\,X=KX,
   \end{equation}
   where $\text{div}$ is the divergence operator of $\Sigma$ and $\Delta^{\Sigma}=\sum_{i=1}^2\{\nabla_{e_i}\nabla_{e_i}-\nabla_{\nabla_{e_i}e_i}\}$ is the rough Laplacian of $\Sigma$, being $\{e_1,e_2\}$ an orthonormal reference in $\Sigma$.

For any non-null vector $a\in\r^8$, from \eqref{eq:harmonic-vector-field} it follows that
\[
\text{div}\,(\langle\hat{\Phi},a\rangle X)=\langle X,a\rangle,
\]
and so $\int_{\Sigma}\langle X,a\rangle\,dA=0$. We are going to use the functions $\langle X,a\rangle,\,a\in\r^8$ as test functions, i.e., we consider the vectorial function $X:\Sigma\rightarrow\r^8$ as a test function. Now the stability of $\Sigma$ implies that the quadratic form $Q$ satisfies $Q(X)\geq 0$ for any harmonic vector field $X$ on $\Sigma$. Now we compute $Q(X)$.

If $\bar{\nabla}$ is the connection of $\r^8$ and $\{e_1,e_2\}$ is an orthonormal reference on $\Sigma$, then the Gauss equation of the immersion $\hat{\Phi}$ says
\[
\bar{\nabla}_{e_i}X=\nabla_{e_i}X+\langle\sigma(e_i,X),N\rangle N+\langle\hat{\sigma}(e_i,X),\eta\rangle\eta+\bar{\sigma}(e_i,X),
\]
and hence
\begin{eqnarray*}
\langle\bar{\Delta} X,X\rangle=\langle\sum_{i=1}^2\{\bar{\nabla}_{e_i}\bar{\nabla}_{e_i}-\bar{\nabla}_{\bar{\nabla}_{e_i}e_i}\}X,X\rangle\\=
\langle\Delta^{\Sigma}\,X,X\rangle-\sum_{i=1}^2\{ \langle\sigma(e_i,X),N\rangle^2+\langle\hat{\sigma}(e_i,X),\eta\rangle^2+|\bar{\sigma}(e_i,X)|^2 \}.
\end{eqnarray*}
Using \eqref{eq:harmonic-vector-field}, \eqref{eq:second-ff-berger-M4} and \eqref{eq:second-ff-M4-R8} we obtain that
\begin{equation}\label{eq:laplacian-harmonic-vector-field}
\begin{split}
\langle\bar{\Delta}\,X,X\rangle&=K|X|^2-|AX|^2-\alpha|X|^2\\&-\frac{\alpha-1}
{\alpha}(2\alpha+(\alpha-1)(1-C^2))\langle X,\xi\rangle^2-(1-\alpha)(5-C^2)|X|^2,
\end{split}
\end{equation}
where $A$ is the Weingarten endomorphism of $\Phi$ associated to the normal field $N$.

If $J_0$ is the complex structure on the Riemann surface $\Sigma$, it is clear that $X^*=J_0X$ is another harmonic vector field on $\Sigma$ and outside of the zeros of $X$, $\{X,X^*\}$ is a ortogonal reference of $T\Sigma$ with $|X|=|X^*|$. Then we have that
\begin{eqnarray*}
|AX|^2+|AX^*|^2=|\sigma|^2|X|^2,\\
\langle X,\xi\rangle^2+\langle X^*,\xi\rangle^2=(1-C^2)|X|^2.
\end{eqnarray*}
Now from equation \eqref{eq:laplacian-harmonic-vector-field} and taking into account the above relations we obtain that
\begin{equation}\label{eq:laplacian-harmonic-vector-field-dual}
\begin{split}
\langle\bar{\Delta} X,X\rangle+\langle\bar{\Delta} X^*,X^*\rangle&=2K|X|^2-|\sigma|^2|X|^2-(8-6\alpha)|X|^2\\
&-\frac{(\alpha-1)^2}{\alpha}(1-C^2)^2|X|^2.
\end{split}
\end{equation}
Using the Gauss equation \eqref{eq:gauss-equation} and \eqref{eq:laplacian-harmonic-vector-field-dual} we finally get
\[
Q(X)+Q(X^*)=\int_{\Sigma}\left(-4H^2-4\alpha+\frac{(\alpha-1)^2}{\alpha}(1-C^2)^2\right)|X|^2dA.
\]
As $C^2\geq 0$, $1-3\alpha\leq 0$ and $\Sigma$ is stable, we finally obtain that
\[
0\leq\int_{\Sigma}\left(-4H^2+\frac{(1-3\alpha)(1+\alpha)}{\alpha}\right)|X|^2dA
\leq\int_{\Sigma}-4H^2|X|^2dA.
\]
Hence we have got that if $\Sigma$ is stable and the genus $g\geq 1$, then $C=0$, $\alpha=1/3$ and $H=0$. But $C=0$ means that the Killing field $\xi$ is tangent to $\Sigma$ and hence it is parallel. So $\Sigma$ is flat, and from Proposition \ref{prop:CMC-flat-torus} $\Sigma$ is a finite cover of the Clifford torus $\mathcal{T}_{1/3}(0)$ in $\s^3_{1/3}$, and so $\Sigma$ is the Clifford torus. This finishes the proof.
\end{proof}
\begin{remark}
{\rm As in Remark~\ref{remark:new-proof-BCE}, the idea developed in the proof of Theorem $7$ can be used to give another new proof of the Barbosa, DoCarmo and Eschenburg result (Theorem \ref{thm:stability-spheres-Nil3-Sl}, 1) \cite{BCE}). In fact, we consider the round sphere $\s^3$ as a umbilical hypersurface of $\r^4$. Then, if $\Phi:\Sigma\rightarrow \s^3$ is an orientable compact stable constant mean curvature surface of genus $g\geq 1$, we take a pair of harmonic vector fields $X,X^*:\Sigma\rightarrow\r^4$, which satisfy $\int_{\Sigma}X\,dA=\int_{\Sigma}X^*\,dA=0$, as test functions. Following the proof of Theorem $7$ we obtain in this case that
 \[
\langle\bar{\Delta} X,X\rangle+\langle\bar{\Delta} X^*,X^*\rangle=2K|X|^2-|\sigma|^2|X|^2-2|X|^2.
\]
 and then
 \[
0\leq Q(X)+Q(X^*)\leq\int_{\Sigma}\left(-4H^2-4\right)|X|^2dA<0.
\]
This contradiction says that the genus of an orientable compact stable constant mean curvature surface of $\s^3$ is zero and hence the surface $\Sigma$ is a umbilical sphere.}
\end{remark}

\section{The isoperimetric problem}
 The isoperimetric problem can be stated as follows: Given a Berger sphere $\s^3_{\alpha}$ and a number $V$ with $0< V<\text{volume}\,(\s^3_{\alpha})=2\pi^2\sqrt{\alpha}$, find the embedded compact surfaces of least area enclosing a domain of volume $V$. In this setting ($\s^3_{\alpha}$ is compact) the problem has always a smooth compact solution, which is a stable constant mean curvature surface.

As in Theorem 7 we have classified the orientable compact stable constant mean curvature surfaces of the Berger spheres $\s^3_{\alpha},\,1/3\leq \alpha<1$, we can solve the isoperimetric problem in these $3$-manifolds.
\begin{quote}
{\it  The solutions of the isoperimetric problem in $\s^3_{\alpha},\,1/3\leq\alpha<1$ are the spheres $\{\mathcal{S}_\alpha(H)\, | \, H \geq 0\}$. }
\end{quote}
In fact, From Theorem 7 the above result is clear when $1/3<\alpha<1$. When $\alpha=1/3$, besides the spheres, the Clifford torus is the only stable constant mean curvature surface. The Clifford torus divides the sphere in two domains of the same volume $\pi^2\sqrt{\alpha}$. Now, among the constant mean curvature spheres of $\s^3_{1/3}$, only the minimal one $\mathcal{S}_{1/3}(0)$ divides the sphere in two domains of the same volume $\pi^2\sqrt{\alpha}$. Since the area of the Clifford torus is $A_1=2\pi^2/\sqrt{3}$, the area of $\mathcal{S}_{1/3}(0)$ is $A_2=2\pi(1+\frac{1}{\sqrt{6}}\arctanh\,(\sqrt{2}/\sqrt{3}))$ and $A_1>A_2$, we finish the proof.

For $\alpha>1$ we think that the spheres $\mathcal{S}_\alpha(H)$ are not only the solutions to the isoperimetric problem but the only compact stable constant mean curvature surfaces in $\s^3_{\alpha}$.

When $\alpha<1/3$, the problem seems to be quite different, because on the one hand there are unstable constant mean curvature spheres and on the other hand there are examples of stable constant mean curvature tori. To illustrate the isoperimetric problem in this case ($\alpha<1/3$), in figure $4$ we have drawn the area of the spheres  $\mathcal{S}_\alpha(H)$ and the tori $\mathcal{T}_\alpha(H)$ in terms of their volumes for four different Berger spheres (see \cite{Tr}).

 \begin{figure}[htbp]
        \begin{minipage}[c]{0.4\textwidth}
         \centering
                \includegraphics[width=\textwidth]{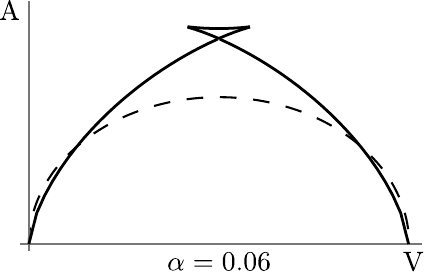}
        \end{minipage}
        \hspace{0.1\textwidth}
        \begin{minipage}[c]{0.4\textwidth}
         \centering
        \includegraphics[width=\textwidth]{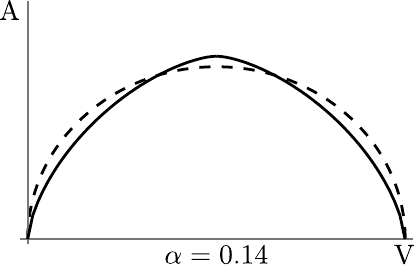}
        \end{minipage}
        \\[0.5cm]
        \begin{minipage}[c]{0.4\textwidth}
         \centering
                \includegraphics[width=\textwidth]{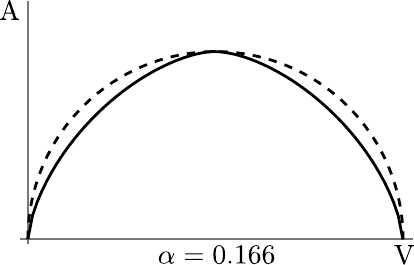}
        \end{minipage}
        \hspace{0.1\textwidth}
        \begin{minipage}[c]{0.4\textwidth}
         \centering
                \includegraphics[width=\textwidth]{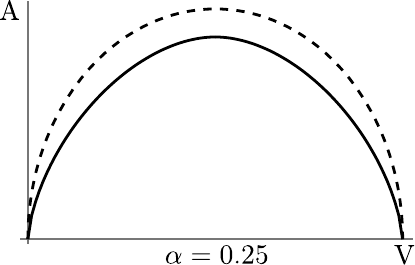}
        \end{minipage}
        \caption{Graphics of the area of $\mathcal{S}_\alpha(H)$ (solid line) and $\mathcal{T}_\alpha(H)$ (dashed line) in terms of the volume for different Berger spheres}
 \end{figure}
When $\alpha=0,25<1/3$, the spheres $\mathcal{S}_\alpha(H)$ are the candidates to solve the isoperimetric problem. But we can find a Berger sphere $\s^3_{\alpha}$ ($\alpha\approx 0,166$) for which the minimal sphere $\mathcal{S}_\alpha(0)$ and the Clifford torus $\mathcal{T}_\alpha(0)$, which divide the Berger sphere in to domains of the same volume, have the same area. So both are candidates to solve the isoperimetric problem. When $\alpha=0,14$ all the spheres $\mathcal{S}_\alpha(H)$ are stable because $\alpha_1<0,14$ (use Theorem \ref{thm:stability-spheres}), but there is an interval of volumes centered at $\pi^2\sqrt\alpha$ for which the tori $\mathcal{T}_\alpha(H)$ are the candidate to solve the isoperimetric problem. Finally, when $\alpha=0,06$ (in this case there are unstable spheres), the tori $\mathcal{T}_\alpha(H)$ are again candidate to solve the isoperimetric problem when the volume is neither close to $0$ nor $2\pi^2\sqrt\alpha$. In this case there are noncongruent spheres enclosing the same volume.

\end{document}